\documentclass[11pt,english]{article}
\usepackage[left=2.5cm, right=2.5cm, top=2cm, bottom=2cm]{geometry}
\usepackage[colorlinks=true,linkcolor=black,citecolor=black,filecolor=black,linkcolor=black,urlcolor=black]{hyperref}
\usepackage{mathrsfs}
\usepackage{relsize}
\usepackage[T1]{fontenc}
\usepackage{lmodern}
\usepackage[utf8]{inputenc}
\usepackage{indentfirst}
\usepackage{nomencl}
\usepackage{color}
\usepackage{graphicx}
\usepackage{microtype}
\usepackage{lipsum}
\usepackage{multicol}										
\usepackage{mathtools}
\usepackage{latexsym}				
\usepackage{amssymb,amsmath,amsthm,amsfonts,amscd,amstext,amsbsy}
\usepackage{enumerate}
\usepackage{todonotes}
\usepackage{enumitem}

\DeclareMathOperator{\conex}{\ds\mathlarger{\mathlarger{\nabla}}}

\newcommand{\Q}{\mathfrak{Q}}
\newcommand{\Hc}{\mathcal{H}}

\newcommand{\M}{\mathcal{M}}
\DeclareMathOperator{\diver}{div}

\DeclareMathOperator{\Hess}{\nabla^2}
\newcommand{\Hessij}{\nabla^2_{ij}}
\DeclareMathOperator{\dist}{dist}
\DeclareMathOperator{\sech}{sech}
\DeclareMathOperator{\diam}{diam}
\DeclareMathOperator{\R}{\mathbb{R}}
\DeclareMathOperator{\HH}{\mathbb{H}}
\DeclareMathOperator{\Ricc}{Ricc}
\newcommand{\W}{\Omega}
\newcommand{\norm}[1]{\left\| #1 \right\|}
\newcommand{\modulo}[1]{\left| #1 \right|}
\usepackage{mathrsfs}
\newcommand{\cl}{\mathscr{C}}
\newcommand{\ds}{\displaystyle}

\newcommand{\Eum}{\ds\mathsmaller{\frac{\partial}{\partial x_1}}}

\newcommand{\Ei}{\ds\mathsmaller{\frac{\partial}{\partial x_i}}}
\newcommand{\Ej}{\ds\mathsmaller{\frac{\partial}{\partial x_j}}}
\newcommand{\Ek}{\ds\mathsmaller{\frac{\partial}{\partial x_k}}}

\newcommand{\Dz}{\partial_z}

\newcommand{\Di}{\partial_i}
\newcommand{\Dj}{\partial_j}
\newcommand{\Dk}{\partial_k}

\newcommand{\Dij}{\partial_{ij}}
\newcommand{\Dii}{\partial_{ii}}
\newcommand{\Dkk}{\partial_{kk}}
\newcommand{\Dkki}{\partial_{kki}}
\newcommand{\Dkkum}{\partial_{kk1}}
\newcommand{\Dki}{\partial_{ki}}
\newcommand{\Dkj}{\partial_{kj}}

\newcommand{\Dum}{\partial_1}
\newcommand{\Dumk}{\partial_{1k}}

\newcommand{\Dumi}{\partial_{1i}}
\newcommand{\Dumum}{\partial_{11}}
\newcommand{\Dumumum}{\partial_{111}}
\newcommand{\Dumij}{\partial_{1ij}}
\newcommand{\Dumii}{\partial_{1ii}}

\newcommand{\escalar}[2]{{\left\langle #1,#2 \right\rangle}}

\newcommand{\lrarrow}{\longrightarrow}
\newcommand{\funciones}[5]{\begin{array}{rccl}
#1:&#2&\lrarrow&#3\\
&#4& \xmapsto{\phantom{\lrarrow}}&#5\end{array}}

\makeatletter
\newcommand{\pushright}[1]{\ifmeasuring@#1\else\omit\hfill$\displaystyle#1$\fi\ignorespaces}
\newcommand{\pushleft}[1]{\ifmeasuring@#1\else\omit$\displaystyle#1$\hfill\fi\ignorespaces}
\makeatother

\newtheorem{teo}{Theorem}

\newtheorem{lema}[teo]{Lemma}

\theoremstyle{definition}

\newtheorem{obs}[teo]{Remark}

\usepackage{amsthm}
\makeatletter
\def\th@plain{%
  \thm@notefont{}
  \itshape 
}
\def\th@definition{%
  \thm@notefont{}
  \normalfont 
}
\makeatother

\makeatletter
\let\oldr@@t\r@@t
\def\r@@t#1#2{%
\setbox0=\hbox{$\:\oldr@@t#1{#2\,}$}\dimen0=\ht0
\advance\dimen0-0.2\ht0
\setbox2=\hbox{\vrule height\ht0 depth -\dimen0}%
{\box0\lower0.4pt\box2}}
\LetLtxMacro{\oldsqrt}{\sqrt}
\renewcommand*{\sqrt}[2][]{\oldsqrt[#1]{#2}}
\makeatother

\newcommand\blfootnote[1]{%
  \begingroup
  \renewcommand\thefootnote{}\footnote{#1}%
  \addtocounter{footnote}{-1}%
  \endgroup
}

\begin{document}
\title{Existence Serrin type results for the Dirichlet problem for the prescribed mean curvature equation in Riemannian manifolds}
\author{Yunelsy N. Alvarez \and	Ricardo Sa Earp}
\maketitle


\begin{abstract}
Given a complete $n$-dimensional Riemannian manifold $M$, we study the existence of vertical graphs in $M\times\mathbb{R}$ with prescribed mean curvature $H=H(x,z)$. 
Precisely, we prove that the Dirichlet problem for the vertical mean curvature equation in a smooth bounded domain $\Omega\subset M$ has solution for arbitrary smooth boundary data if $(n-1)\mathcal{H}_{\partial\Omega}(y)\geq n\sup\limits_{z\in\mathbb{R}}\left|{H(y,z)}\right|$ for each $y\in\partial\Omega $ provided the function $H$ also satisfies $\mathrm{Ricc}_x\geq n\sup\limits_{z\in\mathbb{R}}\left\|\nabla_x H(x,z)\right\|-\dfrac{n^2}{n-1}\inf\limits_{z\in\mathbb{R}}\left(H(x,z)\right)^2$ for each $x\in\Omega$.  
In the case where $M=\mathbb{H}^n$ we also establish an existence result if the condition $\sup\limits_{\Omega\times\mathbb{R}}\left|{H(x,z)}\right|\leq \frac{n-1}{n}$ holds in the place of the condition involving the Ricci curvature. 
Finally, we have a related result when $M$ is a Hadamard manifold whose sectional curvature $K$ satisfies $-c^2\leq K\leq -1$ for some $c>1$. 
We generalize a classical result of Serrin when   the ambient is the Euclidean space. 
\end{abstract}


\section{Introduction}
\markboth{Y. N. Alvarez and R. Sa Earp}{Introduction}
\label{secIntroducao}

Let $\W$ be a smooth bounded domain in a complete Riemannian manifold $M$ of dimension $n\geq 2$. For a prescribed function $H=H(x,z)$ defined on $\overline{\W}\times\R$, we are interest on find a smooth up to the boundary function $u$ satisfying 
\begin{equation}\label{operador_minimo_1}
\diver \left(\dfrac{\nabla u}{\sqrt{1+\norm{\nabla u}^2}}\right) = nH(x,u)
\end{equation}
in $\W$ and taking given values $\varphi$ on $\partial\W$. 

The quantities involved in \eqref{operador_minimo_1} are calculated with respect to the metric of $M$. 
In a coordinate system $(x_1,\dots,x_n)$ equation \eqref{operador_minimo_1} can be written in non-divergence form as 
\begin{equation}\label{operador_minimo_1_coord}
\M u:=\sum_{i,j=1}^n \left(W^2\sigma^{ij} - {u^iu^j}\right)\Hessij u=nH(x,u){W^3} ,
\end{equation}
where $W=\sqrt{1+\norm{\nabla u (x)}^2}$, $(\sigma^{ij})$ is the inverse of the metric $(\sigma_{ij})$ of $M$, $u^i=\ds\sum_{j=1}^n\sigma^{ij} \Dj u$ are the coordinates of $\nabla u$ and $\Hessij u(x)=\Hess u(x){\left(\Ei,\Ej\right)}$. 
We also define the operator $\Q$ by
\begin{equation}\label{oper_Q}
\Q u := \M u -n H(x,u)W^3.
\end{equation}

The matrix of the operator $\M$ (and $\Q$) is given by $A={W^2}g$, where $g$ is the induced metric on the graph of $u$. This implies that the eigenvalues of $A$ are positive and depend on $x$ and $\nabla u$. Hence, $\M$ is locally uniformly elliptic. Furthermore, if $\W$ is bounded and $u\in\cl^{1}(\overline{\W})$, then $\M$ is uniformly elliptic in $\overline{\W}$ (see \cite{spruck} for more details).

We recall that the Dirichlet problem for equation \eqref{operador_minimo_1} is a classical problem in the intersection between Differential Geometry and Partial Differential Equations. First steps were given by Bernstein \cite{Bernstein}, Douglas \cite{Douglas1931} and Rad\'o \cite{Rado1930} 
in domains of $\R^2$ for the minimal case. In 1966 Jenkins-Serrin {\cite[Th. 1 p. 171]{Serrin1968}} derived a related result in higher dimensions. 


Later on, Serrin \cite{Serrin} devoted his attention to study Dirichlet problems for a class of more general elliptic equations within which is the prescribed mean curvature equation. Specifically related to our work, he obtained the following result.
\begin{teo}[Serrin {\cite[Th. p. 484]{Serrin}}]\label{T_Serrin_Ricci}
Let $\W\subset\R^n$ be a bounded domain whose boundary is of class $\cl^2$. Let $H(x)\in\cl^1(\overline{\W})$ and suppose that 
\begin{equation}\label{cond_Ricc_Serrin}
\modulo{\nabla H(x)}\leq \dfrac{n}{n-1}(H(x))^2\ \forall\ x\in\W.
\end{equation}
Then the Dirichlet problem in $\W$ for surfaces having prescribed mean curvature $H(x)$ is uniquely solvable for arbitrarily given $\cl^2$ boundary values if, and only if,
\begin{equation}\label{SerrinCondition}
(n-1)\Hc_{\partial\W}(y)\geq n\modulo{H(y)} \ \forall \ y\in\partial\W,
\end{equation}
where $\Hc_{\partial\W}$ denotes the inward mean curvature of $\partial\W$. 
\end{teo}

Since \eqref{cond_Ricc_Serrin} is satisfied by every $H\in\R$, it follows that the \textit{Serrin condition} \eqref{SerrinCondition} is a necessary and sufficient condition for graphs with constant mean curvature to exist over bounded domains of the Euclidean ambient space. Actually, the following classical sharp solvability criteria holds.
\begin{teo}[Serrin {\cite[p. 416]{Serrin}}]\label{SharpSerrin} 
Let $\W\subset\R^n$ be a bounded domain whose boundary is of class $\cl^2$. Then for every constant $H$ the Dirichlet problem for equation $\M u=nH$ in $\W$ has a unique solution for arbitrary $\cl^2$ boundary data if, and only if, 
$(n-1) \Hc_{\partial\W} \geq n \modulo{H}$.
\end{teo}

The main goal of the present paper is to generalize to more general ambient spaces the existence part in Theorem \ref{T_Serrin_Ricci}. Specifically, we prove the following. 
\begin{teo}[Main theorem]\label{T_exist_Ricci}
Let $\Omega \subset M$ be a bounded domain with $\partial\W$ of class $\cl^{2,\alpha}$ for some $\alpha\in(0,1)$. 
Let $H\in\cl^{1,\alpha}(\overline{\W}\times\R)$ satisfying $\Dz H \geq 0$ and
\begin{equation}\label{cond_H_Ricci_exist}
\Ricc_x\geq n\sup\limits_{z\in\R}\norm{\nabla_x H(x,z)}-\dfrac{n^2}{n-1}\ds\inf_{z\in\R}\left(H(x,z)\right)^2 \ \forall \ x\in\W.
\end{equation}
If 
\begin{equation}\label{StrongSerrinCondition_mainThm}
(n-1)\Hc_{\partial\W}(y)\geq n \sup\limits_{z\in\R}\modulo{H\left(y,z\right)} \ \forall \ y\in\partial\W,
\end{equation}
then for every $\varphi\in\cl^{2,\alpha}(\overline{\W})$ there exists a unique solution $u\in\cl^{2,\alpha}(\overline{\W})$ of the Dirichlet problem for equation \eqref{operador_minimo_1}.
\end{teo}
\noindent In the statement above $\Ricc_x$ is the Ricci curvature\footnote{The definition of the Ricci curvature we use throughout the text follows \cite{petersen1998}. The notation $\Ricc_x \geq f(x)$ means that the Ricci curvature evaluated in any unitary tangent vector at $x$ is bounded below by the function $f(x)$.} of $M$ at $x$. 
Note that relation \eqref{cond_Ricc_Serrin} is a particular case of \eqref{cond_H_Ricci_exist}. 

\medskip
On the other hand, in a previous work 
we have proved that the {\it strong Serrin condition} \eqref{StrongSerrinCondition_mainThm} is sharp in every Hadamard manifold (see \cite[Corollary 2 p. 3]{artigonaoexist}). The combination of this non-existence result with Theorem \ref{T_exist_Ricci} fully generalizes Theorem \ref{T_Serrin_Ricci} to every Hadamard manifold in the $\cl^{2,\alpha}$ class (see also \cite[Th. 10 p. 6]{artigonaoexist}).  
We have also proved that \eqref{StrongSerrinCondition_mainThm} is sharp in every compact and simply connected manifold which is strictly $1/4-$pinched\footnote{A Riemannian manifold is said to be strictly $1/4-$pinched if the sectional curvature $K$ of $M$ satisfies $0<\frac{1}{4} K_0 < K \leq K_0$.} provided $\diam(\W)<\frac{\pi}{2\sqrt{K_0}}$ (see \cite[Corollary 3 p. 4]{artigonaoexist}). A further Serrin type solvability criteria is directly obtained by combining this non-existence result with Theorem \ref{T_exist_Ricci}.

Theorem \ref{T_exist_Ricci} also generalizes a result proved by Spruck for $H\in\R$ in the $M\times\R$ setting (see {\cite[T 1.4 p. 787]{spruck}}). In this case of constant mean curvature the Serrin condition by itself is not always sufficient for the solvability of the Dirichlet problem for equation \eqref{operador_minimo_1} as it happens in the Euclidean space. By way of illustrating better this fact notice that when $M=\HH^n$ and $H\in\R$, relation \eqref{cond_H_Ricci_exist} reduces to $\modulo{H}\geq \frac{n-1}{n}$. In the opposite case $\modulo{H}<\frac{n-1}{n}$, Spruck stated an existence theorem assuming that the strict inequality $(n-1) \Hc_{\partial\W} > n \modulo{H}$ holds (see \cite[Th. 5.4 p. 797]{spruck}). In this paper we also extend this result of Spruck in the hyperbolic space including the inequality even for not necessarily constant $H$, as can be seen in the following theorem. 
\begin{teo}\label{exist_hiperbolicoHmenor1}
Let $\W\subset \HH^n$ be a bounded domain with $\partial\W$ of class $\cl^{2,\alpha}$ for some $\alpha\in(0,1)$ and $\varphi\in\cl^{2,\alpha}(\overline{\W})$. Let $H\in\cl^{1,\alpha}(\overline{\W}\times\R)$ satisfying $\Dz H\geq 0$ and $\sup\limits_{\W\times\R}\modulo{H}\leq \frac{n-1}{n}$. 
If 
$$
(n-1)\Hc_{\partial\W}(y)\geq n \sup\limits_{z\in\R}\modulo{H\left(y,z\right)} \ \forall \ y\in\partial\W,
$$
then for every $\varphi\in\cl^{2,\alpha}(\overline{\W})$ there exists a unique solution $u\in\cl^{2,\alpha}(\overline{\W})$ of the Dirichlet problem for equation \eqref{operador_minimo_1}.
\end{teo}

We note that our non-existence result for Hadamard manifolds \cite[Cor. 1 p. 3]{artigonaoexist} guaranties that Theorem \ref{exist_hiperbolicoHmenor1} is sharp (see also \cite[Th. 7 p. 5]{artigonaoexist}). Besides, putting together 
Spruck's existence theorem \cite[Th. 1.4 p. 787]{spruck} one can deduce that the Serrin sharp solvability criterion for constant $H$ stated in Theorem \ref{SharpSerrin} also holds in the hyperbolic space (see also \cite[Th. 8 p. 5]{artigonaoexist}). 

\medskip

At last, we use the barriers constructed by Galvez-Lozano \cite[Th. 6 p. 12]{Galvez2015} to prove the following result in Hadamard manifolds. 
\begin{teo}\label{teo_GalvezLozano}
Let $M$ be a Hadamard manifold with sectional curvature pinched between $-c^2$ and $-1$ for some $c>1$. Let $\W\subset M$ be a bounded domain with $\partial\W$ of class $\cl^{2,\alpha}$ for some $\alpha\in(0,1)$ and whose principal curvatures are greater than $c$. Let $\varphi\in\cl^{2,\alpha}(\overline{\W})$ and $H\in\cl^{1,\alpha}(\overline{\W}\times\R)$ satisfying $\Dz H\geq 0$ and $\sup\limits_{\W\times\R}\modulo{H}\leq \frac{n-1}{n}$. 
Then the Dirichlet problem for equation \eqref{operador_minimo_1} has a unique solution $u\in\cl^{2,\alpha}(\overline{\W})$.
\end{teo}

\paragraph{Acknowledgements.} The authors were partially supported by the Conselho Nacional de Desenvolvimento Científico e Tecnológico - Brasil (CNPq). The first author was also financed in part by the Coordenação de Aperfeiçoamento de Pessoal de Nível Superior - Brasil (CAPES) - Finance Code 001.

\section{The a priori estimates}

In order to prove the theorems stated in the introduction we use the classical Leray-Schauder degree theory. In this section we derive the a priori estimates for the solution of the Dirichlet problem
\begin{equation}\tag{$P$}\label{ProblemaP}
\left\{\begin{split}
 \diver\left(\dfrac{\nabla u}{W}\right)  &=  nH(x,u)\ \mbox{ in }\ \W, \\
 \phantom{\diver\left(\dfrac{\nabla u}{W}\right)}      u  &= \varphi \ \mbox{ in }\  \partial\Omega.
\end{split}\right.
\end{equation}

Since \eqref{operador_minimo_1} and \eqref{operador_minimo_1_coord} are equivalent, the operator $\Q$ defined in \eqref{oper_Q} and the equation $\Q u=0$ are used in this section in order to facilitate the calculations.

Firstly, we establish a lemma that will help us to obtain the a priori height and boundary gradient estimates. 
{
\begin{lema}\label{lema_Est_altura_cond_ricci}
Let $\Gamma$ be an embedded and compact $\cl^2$ hypersurface of $M$ oriented by a global unit normal $N$. Let $\tau>0$ be such that 
$$ \funciones{\Phi_t}{\Gamma}{\Gamma_t\subset M}{x}{\exp^{\bot}(x,tN_x)}$$
is a diffeomorphism between $\Gamma$ and $\Gamma_t$ for each $t\in[0,\tau)$. For $y\in\Gamma$ fixed, let $\gamma_y(t)=\exp_{y}(tN_y)$ with  $0\leq t\leq \tau$. If there exists a function $h\in\cl^1[0,\tau)$ such that
\begin{equation}\label{der_H_3}
\Hc_{\Gamma}(y)\geq \modulo{h(0)}
\end{equation}
and
\begin{equation}\label{der_H_3_2}
\Ricc_{\gamma_y(t)}(\gamma_y'(t))\geq (n-1)\left(\modulo{h'(t)}-(h(t))^2\right)  \  \forall \ t\in[0,\tau),
\end{equation} 
then 
\begin{equation}\label{est_laplaciano_est_altura_1_lema}
\Hc_{\Gamma_t}(\gamma_y(t))\geq \modulo{h(t)} \  \forall \ t\in[0,\tau),
\end{equation} 
where $\Hc_{\Gamma_t}(\gamma_y(t))$ is the mean curvature of $\Gamma_t$ at $\gamma_y(t)$ computed with respect to $\gamma_y'(t)$. Furthemore, $\Hc_{\Gamma_t}(\gamma_y(t))$ is increasing as a function of $t$.
\end{lema}}%
\begin{proof}
First of all note that the hypersurface $\Gamma_t$ is  parallel to $\Gamma$ for each $t\in[0,\tau)$. Let $\Hc(t):=\Hc_{\Gamma_t}(\gamma_y(t))$. It is well known that (see \cite[Th. 2.7 p. 3]{Barbosa1987} and \cite[Cor. B.14 p. 66]{minhatese} for a different proof)
\begin{equation}\label{nextRemark}
\Hc'(t) \geq  \dfrac{\Ricc_{\gamma_y(t)}(\gamma_y'(t))}{n-1} + \left(\Hc(t)\right)^2.
\end{equation}
From \eqref{der_H_3_2} it follows 
\begin{equation}\label{desig_sem_ricc}
\Hc'(t) \geq \modulo{h'(t)} - (h(t))^2 + \left(\Hc(t)\right)^2.
\end{equation}
Then,
\begin{equation}\label{desig_sem_ricc_v}
(\Hc(t) - h(t))'\geq \left(\Hc(t)+h(t)\right)\left(\Hc(t)-h(t)\right)
\end{equation}
and
\begin{equation}\label{desig_sem_ricc_g}
(\Hc(t) + h(t))'\geq \left(\Hc(t)-h(t)\right)\left(\Hc(t)+h(t)\right).
\end{equation}
Let us define $v(t)=\Hc(t)-h(t)$ and $g(t)=\Hc(t)+h(t)$. 
Inequality \eqref{desig_sem_ricc_v} yields
$$\left(\dfrac{v(t)}{\ds e^{\int_0^t g(s)ds}}\right)'\geq 0,$$
so 
$v(t) \geq\ds v(0)e^{\int_0^t g(s)ds}$ for each $t\in[0,\tau)$. As a consequence of \eqref{der_H_3}, 
$$\Hc(t)\geq h(t) \ \forall t\in[0,\tau).$$

Using \eqref{desig_sem_ricc_g} it follows in a similar way 
$$\Hc(t)\geq -h(t) \ \forall t\in[0,\tau).$$
Therefore, 
\begin{equation}\label{H_h}
\Hc(t)\geq\modulo{h(t)} \ \forall t\in[0,\tau).
\end{equation}
Finally, $\Hc'(t)\geq 0$ is obtained by substituting \eqref{H_h} in \eqref{desig_sem_ricc}. 
\end{proof}

\begin{obs}\label{Remark_parallel}
By choosing the function $h$ appropriately condition \eqref{der_H_3} becomes the strong Serrin condition \eqref{StrongSerrinCondition_mainThm} and condition \eqref{der_H_3_2} turns \eqref{cond_H_Ricci_exist}. 
Roughly speaking, Lemma \ref{lema_Est_altura_cond_ricci} says that the parallel hypersurfaces to $\Gamma$ inherit the initial condition \eqref{der_H_3} throughout the orthogonal geodesics provided condition \eqref{der_H_3_2} also holds. Moreover, if such a function $h$ exists, necessarily the mean curvature of the parallel hypersurfaces lying inside $\W$ increases along the inner normal geodesics. 
This property is true in mean convex domains of manifolds whose Ricci curvature is non-negative but is not necessarily true in the opposite case (see \eqref{nextRemark}).
\end{obs}

%


\begin{teo}[A priori height estimate]\label{teo_Est_altura}
Let $\W\in M$ be a bounded domain with $\partial\W$ of class $\cl^2$ and $\varphi\in\cl^0(\partial\W)$. Let $H\in\cl^{1}(\overline{\W}\times\R)$ satisfying $\Dz H\geq 0$,
\begin{equation}\label{cond_H_Ricci_sup}
\Ricc_x\geq n\sup\limits_{z\in\R}\norm{\nabla_x H(x,z)}-\dfrac{n^2}{n-1}\inf\limits_{z\in\R}\left(H(x,z)\right)^2 \ \forall \ x\in\W,
\end{equation}
and
\begin{equation}\label{cond_Serrin_hightest_teo}
(n-1)\Hc_{\partial\W}(y)\geq n \modulo{H(y,\varphi(y))} \ \forall \ y\in\partial\W.
\end{equation}
If $u\in\cl^2(\W)\cap\cl^0(\overline{\W})$ is a solution of the Dirichlet problem \eqref{ProblemaP}, then 
\begin{equation*}
\sup\limits_{\W}\modulo{u}\leq \sup\limits_{\partial\W} \modulo{\varphi} +\dfrac{e^{\mu\delta}-1}{\mu},
\end{equation*}
where $\mu>n\sup\left\{\modulo{H(x,z)}, (x,z)\in \overline{\W}\times\left[-\sup\limits_{\partial\W}\modulo{\varphi},\sup\limits_{\partial\W}\modulo{\varphi}\right]\right\}$ and $\delta=\diam(\W)$.
\end{teo}
\begin{proof}
For $x\in\W$ let us define the distance function $d(x)=\dist(x,\partial\W)$. Let $\W_0$ be the biggest open subset of $\W$ having the unique nearest point property; that is, for every $x\in\W_0$ there exists a unique $y\in\partial\W$ such that $d(x)=\dist(x,y)$. 
Then $d\in\cl^2(\W_0)$ (see \cite{LiNirenberg2}).

Let 
$$\phi(t)=\dfrac{\ds e^{\mu\delta}}{\mu}\left(1-e^{-\mu t}\right)$$
where $\mu$ and $\delta$ are the constant defined in the statement of the theorem. We now define \linebreak $w=\phi\circ d + \sup\limits_{\partial\W}\modulo{\varphi}$.
The desired estimate follows if $\modulo{u}\leq w$ in $\overline{\W}$. First, we will prove that $u\leq w$. By contradiction suppose that $u-w$ attains a maximum $m>0$ at $x_0\in{\W}$. That is,
\begin{equation}\label{uleqwm}
u\leq w + m \mbox{ in } \W
\end{equation}
and 
\begin{equation}\label{uxowm}
u(x_0) = w(x_0) + m .
\end{equation}

Choose $y_0\in\partial\W$ such that $d(x_0)=\dist(x_0,y_0)$ and let $\gamma$ be the minimizing geodesic orthogonal to $\partial\W$ joining $x_0$ to $y_0$. Then, $d(x)=\dist(x,y_0)$ for every $x$ lying in $\gamma$ between $x_0$ and $y_0$. By restricting $u$ and $w$ to $\gamma$ we see that $u'(d(x_0))=w'(d(x_0))=\phi'(d(x_0))>0$. Thus, $\nabla u(x_0)\neq 0$ and the level set
$\Gamma_0=\left\{x\in\W;u(x)=u(x_0)\right\}$ is an hypersurface of class $\cl^2$ in a neighbourhood of $x_0$. Consequently, there exists a geodesic ball $B_{\epsilon}(z_0)$ tangent to $\Gamma_0$ in $x_0$ such that 
\begin{equation}\label{eq_bola_1}
u\geq u(x_0) \mbox{ in } \overline{B_{\epsilon}(z_0)}. 
\end{equation}

From \eqref{uleqwm}, \eqref{uxowm} and \eqref{eq_bola_1} we obtain for $x\in\overline{B_{\epsilon}(z_0)}$,
$$w(x_0)+m=u(x_0)\leq u(x)\leq w(x)+m. $$
Hence, $w(x_0)\leq w(x)$ which yields 
\begin{equation}\label{eq111}
d(x_0)\leq d(x) \mbox{ in } \overline{B_{\epsilon}(z_0)}
\end{equation} 
since $\phi$ is increasing. 

Therefore, if $\overline{x}$ lies in the intersection of $\partial B_{\epsilon}(z_0)$ with a geodesic minimizing the distance between $z_0$ and $y_0$, then
\begin{equation*}
\begin{split}
\dist(z_0,y_0)=&\dist(z_0,\overline{x})+\dist(\overline{x},y_0)\\
 =&\dist(z_0,{x_0})+\dist(\overline{x},y_0)\\
 \geq & \dist(z_0,{x_0})+d(\overline{x})\\
 \geq &\dist(z_0,{x_0})+d({x_0})\mbox{ from \eqref{eq111}}\\
  =&\dist(z_0,{x_0})+d({x_0},y_0).
\end{split}
\end{equation*}
Thus, equality holds in the triangle inequality which implies that $x_0=\overline{x}$. That is, $z_0=\gamma(d(x_0)+\epsilon)$. This ensures that $x_0\in\W_0$ because if there exists $y_1\neq y_0$ satisfying $d(x_0)=\dist(x_0,y_1)$, then
$$ \dist(z_0,y_1)<\dist(z_0,x_0)+\dist(x_0,y_1)=\dist(z_0,x_0)+\dist(x_0,y_0)=d(z_0),$$ 
which contradicts the definition of $d$.

However, $x_0$ can not be in $\W_0$ as will be shown in the sequel. Some algebraic computations yields
\begin{equation}\label{Mw_est_altura_0}
\M w =  \phi'(1+\phi'^2) \Delta d + {\phi''} \ \mbox{ in } \W_0.
\end{equation}
For $x\in\W_0$, let $y=y(x)$ in $\partial\W$ be the nearest point to $x$ and $\gamma_y(t)$ the orthogonal geodesic to $\partial\W$ from $y$ to $x$. Let us define 
$$h(t)=\frac{n}{n-1}H\left(\gamma_y(t),\varphi(y)\right).$$
Note that $y$ is now fixed. 
On account of the Serrin condition \eqref{cond_Serrin_hightest_teo} one has
$$\modulo{h(0)} = \dfrac{n}{n-1} \modulo{H\left(y,\varphi(y)\right)}\leq \Hc_{\partial\W}(y) =\Hc(0) .$$
Besides, 
$$ h'(t)=\dfrac{n}{n-1}\escalar{\nabla_x H(\gamma_y(t),\varphi(y))}{\gamma_y'(t)}. $$
The additional hypothesis \eqref{cond_H_Ricci_sup} yields
$$(n-1)\left(\modulo{h'(t)}-(h(t))^2\right)\leq \Ricc_{\gamma_y(t)}(\gamma_y'(t)).$$
Thus, Lemma \ref{lema_Est_altura_cond_ricci} can be applied to the function $h(t)$ to obtain 
$$ n \modulo{H(\gamma_y(t),\varphi(y))}\leq  (n-1)\Hc_{\Gamma_t}(\gamma_y(t)),$$
where $\Gamma_t$ is parallel to some portion of $\partial\W$.  
By using a well-known formula linking the Laplacian of the distance function with the mean curvature of parallel hypersurfaces, we get
$$\Delta d(x)\leq-n\modulo{H\left(x,\varphi(y(x))\right)} \ \forall \ x\in\W_0.$$
Using this estimate in \eqref{Mw_est_altura_0} it follows
$$\M w \leq  -n\modulo{H\left(x,\varphi(y(x))\right)} {\phi'}{(1+\phi'^2)}  + {\phi''}.$$
Also, 
$$
\phi''(t)=-\mu \ds e^{\mu(\delta-t)}=-\mu\phi'(t)<-n\modulo{H(x,\varphi(y(x)))}\phi'(t) 
$$
and $\phi'\geq 1$, so
\begin{equation}\label{est_Mw_est_alt}
\M w 
< -n\modulo{H\left(x,\varphi(y(x))\right)}{\left(1+\phi'^2\right)^{3/2}}.
\end{equation}
On the other hand, the hypothesis $\Dz H\geq 0$ implies that 
\begin{equation}\label{eq_Hxw}
 \mp H(x,\pm w)
\leq \mp H\left(x,{\varphi(y(x))}\right)\leq \modulo{H\left(x,{\varphi(y(x))}\right)}.
\end{equation}
From \eqref{est_Mw_est_alt} and \eqref{eq_Hxw} we conclude that 
\begin{align*} 
\pm \Q(\pm w) = & \M w \mp nH\left(x,\pm w\right) {\left(1+\phi'^2\right)^{3/2}}\leq 0.
\end{align*}
Therefore,
\begin{align*}
\Q(w+m)=&\M (w+m)-nH(x,w+m){\left(1+\phi'^2\right)^{3/2}}
    \leq 
		\Q w \leq \Q u.
\end{align*}
Moreover, $u \leq w + m$ and $u(x_0)=w(x_0)+m$. The maximum principle implies that $u = w+m$ in $\W_0$ which contradicts the fact that $u<w+m$ in $\partial\W$. This proves that $u\leq w$ in $\overline{\W}$. 

In a similar way it can be proved that $u \geq - w$ in $\overline{\W}$.
\end{proof}
\begin{obs} Instead of condition \eqref{cond_H_Ricci_sup}, the proof shows that it suffice to assume 
\begin{equation*}
\Ricc_{x}\geq n\norm{\nabla_x H(x,\varphi(y))}-\dfrac{n^2}{n-1}\left(H(x,\varphi(y))\right)^2 \ \forall \ x\in\W_0,
\end{equation*}
where 
$y\in\partial\W$ is the nearest point to $x$.
\end{obs}


\begin{teo}[Boundary gradient estimate]\label{teo_Est_gradiente_fronteira}
Let $\W\in M$ be a bounded domain with $\partial\W$ of class $\cl^2$ and $\varphi\in\cl^2(\overline{\W})$. 
Let 
$H\in\cl^{1}\left(\overline{\W}\times\R\right)$ satisfying $\Dz H\geq 0$,
\begin{equation}\label{cond_H_Ricci_sup_est_grad_fronteira}
\Ricc_x\geq n\sup\limits_{z\in\R}\norm{\nabla_x H(x,z)}-\dfrac{n^2}{n-1}\inf\limits_{z\in\R}\left(H(x,z)\right)^2 \ \forall \ x\in\W,
\end{equation} 
and
\begin{equation}\label{cond_Serrin}
(n-1)\Hc_{\partial\W}(y)\geq n \modulo{H(y,\varphi(y))} \ \forall \ y\in\partial\W.
\end{equation}
If $u\in\cl^2(\W)\cap\cl^1(\overline{\W})$ solves problem \eqref{ProblemaP}, then
\begin{equation}\label{EstGradFront}
\sup\limits_{\partial\W}\norm{\nabla u}\leq \norm{\varphi}_1 + \ds e^{C\left(1+ \norm{H}_1+ \norm{\varphi}_2\right)\left(1+\norm{\varphi}_1\right)^3\left(\norm{u}_0+\norm{\varphi}_0\right)}
\end{equation}
for some $C=C(n,\W)$.
\end{teo}
\begin{proof}
Again we set $d(x)=\dist(x,\partial\W)$ for $x\in\W$. Let $\tau>0$ be such that $d$ is of class $\cl^2$ over the set of points in $\W$ for which $d(x)\leq\tau$. 
Let $\psi\in\cl^2([0,\tau])$ be a non-negative function satisfying 
\begin{multicols}{3}
\begin{enumerate}
\item[P1.] $\psi'(t)\geq 1$,
\item[P2.] $\psi''(t) \leq 0$,
\item[P3.] $t\psi'(t)\leq 1$.
\end{enumerate}
\end{multicols}

For $a<\tau$ to be fixed latter on we consider the set 
$$\W_{a}=\left\{x\in M; d(x)<a \right\} .$$

We now define $w^{\pm}=\pm \psi\circ d +  \varphi$. Firstly, let us estimate $\pm\M w^{\pm}$ in $\W_a$. A straightforward computation yields 
\begin{equation}\label{eq_barreira_superior}
\begin{split}
\pm \M w^{\pm}=\psi'W_{\pm}^2\Delta d -\psi'\Hess d (  \nabla \varphi, \nabla \varphi) + \psi''W_{\pm}^2-\psi''\escalar{\nabla   d }{\pm \psi' \nabla d  +   \nabla \varphi}^2\\
\pm W_{\pm}^2\Delta \varphi \mp   \Hess\varphi(\pm\psi' \nabla d  +  \nabla \varphi,\pm\psi' \nabla d  +  \nabla \varphi),
\end{split}
\end{equation}
where 
$$
W_{\pm}=\sqrt{1+\norm{\nabla w^{\pm}}^2}=\sqrt{1+\norm{\pm \psi'\nabla  d  +   \nabla \varphi}^2}.
$$

Once $\Hess d(x)$ is a continuous bilinear form and $\psi'\geq 1$ we have 
\begin{equation}\label{BBB}
\psi'\modulo{\Hess d ( \nabla \varphi, \nabla \varphi)} \leq \psi'^2 \norm{d}_2\norm{\varphi}_1^2.
\end{equation}

Since $\psi''<0$ and $\escalar{\nabla   d}{\pm\psi' \nabla d  +  \nabla \varphi}^2\leq \norm{\pm\psi' \nabla d + \nabla \varphi}^2$, then 
\begin{equation}\label{AAA}
\psi''W_{\pm}^2-\psi''\escalar{\nabla   d }{\pm\psi' \nabla d  +  \nabla \varphi}^2\leq \psi''.
\end{equation}

Also, $\varphi$ is of class $\cl^2$ in $\W_a$ by hypothesis, so
\begin{align*}
\modulo{\pm W_{\pm}^2\Delta \varphi \mp   \Hess \varphi(\pm\psi' \nabla d  +  \nabla \varphi,\pm\psi' \nabla d  +  \nabla \varphi)} 
& \leq 2 n \norm{\varphi}_2 W_{\pm}^2.
\end{align*}

Note also that
$$
\norm{\pm\psi' \nabla d  +  \nabla \varphi}^2=\left(\psi'^2+2 \psi'\escalar{\pm\nabla d}{\nabla \varphi}+\norm{\nabla \varphi}^2\right)
\leq \left(1+\norm{\varphi}_1\right)^2\psi'^2,
$$
hence
\begin{equation}\label{est_W2}
W_{\pm}^2 \leq  1 + \left(1+\norm{\varphi}_1\right)^2\psi'^2 \leq 2\left(1+\norm{\varphi}_1\right)^2\psi'^2.
\end{equation}
Therefore, 
\begin{equation}\label{CCC}
\modulo{\pm W_{\pm}^2\Delta \varphi \mp   \Hess \varphi(\pm\psi' \nabla d  +  \nabla \varphi,\pm\psi' \nabla d  +  \nabla \varphi)}
\leq 4n\norm{\varphi}_2\left(1+\norm{\varphi}_1\right)^2\psi'^2.
\end{equation}

Substituting \eqref{AAA}, \eqref{BBB} and \eqref{CCC} in \eqref{eq_barreira_superior} it follows
\begin{equation}\label{est_Mwpm}
\pm\M w^{\pm}\leq \psi' W_{\pm}^2 \Delta d + \psi''+ c \psi'^2,
\end{equation}
where
\begin{equation}\label{constantec0}
c=\norm{d}_2\norm{\varphi}_1^2+4n\norm{\varphi}_2\left(1+\norm{\varphi}_1\right)^2.
\end{equation}


On the other hand,
$$ \mp H(x,w^{\pm}(x)) = \mp H(x,\pm\psi(d(x))+ \varphi(x))\leq \mp H(x,\varphi(x))\leq \modulo{H(x, \varphi(x))}$$
since we are assuming that $\Dz H\geq 0$. Thus,
$$\pm \Q_{ } w^{\pm} = \pm \M w^{\pm} \mp n  H(x,w^{\pm})W_{\pm}^3 \leq  \pm \M w^{\pm} + n  \modulo{H(x, \varphi(x))}W_{\pm}^3.$$
Using the estimate \eqref{est_Mwpm} we obtain
\begin{equation}\label{eq_barreira_superior_Q_2}
\begin{array}{r}
\pm \Q_{ } w^{\pm} \leq  \psi'W_{\pm}^2\Delta d + \psi''+ c \psi'^2 + n  \modulo{H(x, \varphi(x))}W_{\pm}^{3}.
\end{array}
\end{equation}

We now want to estimate $\Delta d$. Let $y\in\partial\W$ be fixed and $\gamma_y(t)=\exp_{y}(tN_y)$ for $0\leq t \leq a$, where $N$ is the inner normal field to $\partial\W$. Again, the hypothesis \eqref{cond_H_Ricci_sup_est_grad_fronteira} and \eqref{cond_Serrin} guaranty that the function $h(t)=\frac{n}{n-1} H(\gamma_y(t),\varphi(y))$ satisfies the hypothesis of Lemma \ref{lema_Est_altura_cond_ricci}. Hence, if $\Gamma_t$ is parallel to $\partial\W$, then 
\begin{equation}\label{paraObs}
\Hc_{\Gamma_t}(\gamma_y(t))\geq \Hc_{\partial\W}(y)\mbox{ in } [0,\tau].
\end{equation}
Thus,
\begin{equation}\label{est_usar_hiperb_0}
\Delta d(x) = -(n-1)\Hc_{\Gamma_{d(x)}}(x) \leq -(n-1) \Hc_{\partial\W}(y) \ \forall \ x\in\W_a,
\end{equation}
where $y=y(x)\in\partial\W$ is the nearest point to $x$. Using again the Serrin condition \eqref{cond_Serrin} it follows
\begin{equation}\label{est_usar_hiperb}
\Delta d(x) \leq -n \modulo{H(y,\varphi(y))} \ \forall \ x\in\W_a.
\end{equation}

Substituting \eqref{est_usar_hiperb} in \eqref{eq_barreira_superior_Q_2} we obtain 
\begin{equation}\label{Qw_medio}
\begin{split}
\pm \Q_{ } w^{\pm} \leq &   n \psi'W_{\pm}^2( \modulo{H(x,\varphi(x))} -\modulo{H(y,\varphi(y))})   +n \modulo{H(x, \varphi(x))}W_{\pm}^2 \left(W_{\pm}-\psi'\right) + \psi''+ c \psi'^2   .
\end{split}
\end{equation}

In addition
\begin{equation}\label{FFF}
\modulo{H(x,\varphi(x))}-\modulo{H(y,\varphi(y))}\leq h_1(1+\norm\varphi_1)d(x),
\end{equation}
where
$$h_1=\sup\limits_{\W\times\left[-\sup\limits_{\W}\modulo{\varphi},\sup\limits_{\W}\modulo{\varphi}\right]}\norm{\nabla_{M\times\R} H(x,z)}.$$
From \eqref{est_W2} and \eqref{FFF} one has
\[n  \psi'W_{\pm}^2( \modulo{H(x,\varphi(x))} -\modulo{H(y,\varphi(y))} )\leq  2nh_1\left(1+\norm{\varphi}_1\right)^3 d(x)(\psi'(d(x)))^3.\]
Using the assumption P3 it follows
\begin{equation}\label{Termo2}
\begin{array}{c}
n  \psi'W_{\pm}^2( \modulo{H(x,\varphi(x))} -\modulo{H(y,\varphi(y))} )\leq 2 n h_1 \left(1+\norm{\varphi}_1\right)^3 \psi'^2.
\end{array}
\end{equation}

On the other hand, 
\begin{equation}\label{Wmenospsi_0}
W_{\pm}-\psi'\leq 1+\norm{\pm \psi'\nabla d +\nabla \varphi} -\psi' \leq 1+\norm{\varphi}_1.
\end{equation}
From \eqref{est_W2} and \eqref{Wmenospsi_0} we obtain 
\begin{equation}\label{Wmenospsi}
n \modulo{H(x, \varphi(x))} \left(W_{\pm}-\psi'\right)W_{\pm}^2\leq 2 n h_0\left(1+\norm{\varphi}_1\right)^3\psi'^2,
\end{equation}
where
$$h_0=\sup\limits_{\W\times\left[-\sup\limits_{\W}\modulo{\varphi},\sup\limits_{\W}\modulo{\varphi}\right]}\modulo{H(x,z)}.$$

Substituting \eqref{Termo2} and \eqref{Wmenospsi} in \eqref{Qw_medio} we get 
$$ \pm \Q_{ } w^{\pm} \leq \left(c+2n\norm{H}_{1}\left(1+\norm{\varphi}_1\right)^3\right)\psi'^2+\psi'',$$
where $\norm{H}_1=h_0+h_1$.

Remembering the expression for $c$ given in \eqref{constantec0} and making some algebraic computation we infer that 
$$ \pm \Q_{ } w^{\pm} < \nu\psi'^2+\psi'',$$
where
\begin{equation}\label{nu}
\nu= 4n\left(1+\norm{d}_2+1/\tau\right) \left(1+ \norm{H}_1+ \norm{\varphi}_2\right)\left(1+\norm{\varphi}_1\right)^3.
\end{equation}
Defining $\psi$ explicity by 
$$
\psi(t)=\dfrac{1}{\nu}\log(1+kt)
$$
we obtain $\nu\psi'^2+\psi''=0$. Indeed, 
\begin{equation}\label{dpsi}
\psi'(t)=\dfrac{k}{\nu(1+kt)}
\end{equation}
and 
\begin{equation}\label{ddpsi}
\psi''(t)=-\dfrac{k^2}{\nu(1+kt)^2}. 
\end{equation}
Therefore,
$$
\pm \Q w^{\pm} < 0 \ \mbox{ in } \ \W_a.
$$

Note that property P2 follows from \eqref{ddpsi}. Another consequence of \eqref{ddpsi} is that $\psi'(t)>\psi'(a)$ for all $t\in[0,a]$, thus property P1 is ensured provided 
\begin{equation}\label{paraP1}
\psi'(a)=\dfrac{k}{\nu(1+ka)}=1.
\end{equation}
Also,
$$ t\psi'(t)=\dfrac{kt}{\nu(1+kt)}\leq \dfrac{1}{\nu}<1$$
which is property P3. Hence, $\psi$ thus defined satisfies all the initial requirements. 

Furthermore, choosing 
\begin{equation}\label{esc_psia_curv_media}
\psi(a) = \dfrac{1}{\nu}\log(1+ka) =  \norm{u}_0+\norm{\varphi}_0
\end{equation}
it follows 
$$\pm w^{\pm}(x)=\psi(a)\pm \varphi(x)=\norm{u}_0+\norm{\varphi}_0 \pm \varphi(x) \geq \pm u(x) \  \forall \ x\in \partial\W_a\setminus\partial\W.$$ 
Besides, for $x\in\partial\W$ one has $w^{\pm}(x)=\pm \psi(0)+\varphi(x)=u(x)$. By the maximum principle it can be conclude that $w^-\leq u \leq w^+ $ in $\W_a$. Thus, 
$$ -\psi\circ d  \leq u - \varphi \leq \psi\circ d \mbox{ in } \W_a, $$
being that  
$$ -\psi\circ d  = u - \varphi = \psi\circ d =0 \mbox{ in } \partial\W.$$
Consequently, if $y\in\partial\W$ and $0 \leq t \leq a$, then
$$-\psi(t) + \psi(0) \leq (u-\varphi) (\gamma_y(t)) - (u-\varphi)(\gamma_y(0)) \leq \psi(t)-\psi(0).$$
Dividing by $t>0$ and passing to the limit as $t$ goes to zero we infer that
\begin{equation}\label{est_grad_fronteira_1}
\modulo{\escalar{\nabla u(y)}{N}}\leq \modulo{\escalar{\nabla \varphi(y)}{N}} + \psi'(0).
\end{equation}
Since $u=\varphi$ on $\partial\W$, we derive from \eqref{est_grad_fronteira_1}
\begin{align*}
\norm{\nabla u(y)}\leq&\norm{\nabla \varphi(y)} + \psi'(0).
\end{align*}

The desired estimate \eqref{EstGradFront} follows from this last expression because the combination of \eqref{paraP1} and \eqref{esc_psia_curv_media} yields
\begin{equation}\label{cte_k}
k=\nu\ds e^{\nu(\norm{u}_0+\norm{\varphi}_0)}.
\end{equation}
Observe also that from \eqref{nu}, \eqref{paraP1} and \eqref{cte_k} it follows
$$
a
=\dfrac{e^{\nu(\norm{u}_0+\norm{\varphi}_0)}-1}{\nu\ds e^{\nu(\norm{u}_0+\norm{\varphi}_0)}}<\frac{1}{\nu}<\tau
$$
as required at the beginning. 
\end{proof}


\bigskip 

Recall now that assumption \eqref{cond_H_Ricci_sup_est_grad_fronteira} was also requested in the statement of Theorem \ref{teo_Est_gradiente_fronteira} because its combination with \eqref{cond_Serrin} ensures, in addition, the geometric property \eqref{paraObs} (see Remark \ref{Remark_parallel}). In order to see that this property does not always happens in mean convex domains of manifolds whose Ricci curvature is non-positive let us consider a mean convex domain $\W$ in the hyperbolic space $\HH^n$. For $y\in\partial\W$ let $\lambda_i(t)$ be the ith principal curvature of $\Gamma_t$ at $\gamma_y(t)$, then (see \cite[p. 17]{marcos})
\begin{equation}\label{curv_princ_paral_explicito}
\lambda_i(t)=\dfrac{-\tanh t+\lambda_i(0)}{1- \lambda_i(0)\tanh t }, 
\end{equation}
hence
\begin{equation}\label{der_curv_princ_hiperb} 
\lambda_i'(t)=\dfrac{\sech^2(t)\left(\left(\lambda_i(0)\right)^2-1\right)}{\left(1-\lambda_i(0)\tanh t\right)^2}.
\end{equation}
Thus, $\Hc_{\Gamma_t}(\gamma_y(t))$ decrease if $\modulo{\lambda_i(0)} < 1$ for all $1\leq i \leq n$. However, if this is the case, 
$\tau$ can be chosen small enough such that 
$$ \modulo{\Hc_{\partial\W}(y)-\Hc_{d(x)}(x)}\leq \kappa d(x)$$
for some $\kappa>0$ depending on $\W$. 
Using this fact we deduce the following result. 
\begin{teo}[Boundary gradient estimate - the hyperbolic case]\label{teo_Est_gradiente_fronteira_hiperbólico}
Let $\W\in \HH^n$ be a bounded domain with $\partial\W$ of class $\cl^2$ and $\varphi\in\cl^2(\overline{\W})$. 
Let $H\in\cl^{1}\left(\overline{\W}\times\R\right)$ satisfying $\Dz H\geq 0$,
and
$$(n-1)\Hc_{\partial\W}(y)\geq n \modulo{H(y,\varphi(y))} \ \forall \ y\in\partial\W.$$
If $u\in\cl^2(\W)\cap\cl^1(\overline{\W})$ solves problem \eqref{ProblemaP}, then
$$
\sup\limits_{\partial\W}\norm{\nabla u}\leq \norm{\varphi}_1 + \ds e^{C\left(1+ \norm{H}_1+ \norm{\varphi}_2\right)\left(1+\norm{\varphi}_1\right)^3\left(\norm{u}_0+\norm{\varphi}_0\right)}
$$
for some $C=C(n,\W)$.
\end{teo}
\begin{proof}
The proof follows the steps of the proof of Theorem \ref{teo_Est_gradiente_fronteira} with the difference that relation \eqref{est_usar_hiperb} is replaced by
\[
\Delta d(x)\leq -n\modulo{H(y,\varphi(y))} + (n-1)\kappa d(x). \qedhere
\]
\end{proof}

\bigskip

On the other hand, applying lemma \ref{lema_Est_altura_cond_ricci} to the constant function $h(t)=\Hc_{\partial\W}(y)$ (see also {\cite[Th. 1 p. 232]{Dajczer2008}}) it can be seen that condition \eqref{paraObs} is guaranteed by the geometric condition
\begin{equation}\label{eq_est_grad_curv_bordo_Ricc_obs}
\Ricc_{\gamma_y(t)}(\gamma_y'(t))\geq -(n-1)\left(\Hc_{\partial\W}(y)\right)^2 \ \forall \ y\in\partial\W. 
\end{equation} 
Therefore, we state the following result for later reference.

\begin{teo}[Boundary gradient estimate - a particular case]\label{teo_Est_gradiente_fronteira_paraGalvez}
Let $M$ be a complete Riemannian manifold whose Ricci curvature satisfies $\Ricc \geq -(n-1)c^2$ for $c>0$. Let $\W\in M$ be a bounded domain with $\partial\W$ of class $\cl^2$ such that $\Hc_{\partial\W} \geq c$. Let $\varphi\in\cl^2(\overline{\W})$ and
$H\in\cl^{1}\left(\overline{\W}\times\R\right)$ satisfying $\Dz H\geq 0$
and
$$
(n-1)\Hc_{\partial\W}(y)\geq n \modulo{H(y,\varphi(y))} \ \forall \ y\in\partial\W.
$$
If $u\in\cl^2(\W)\cap\cl^1(\overline{\W})$ solves problem \eqref{ProblemaP}, then
$$
\sup\limits_{\partial\W}\norm{\nabla u}\leq \norm{\varphi}_1 + \ds e^{C\left(1+ \norm{H}_1+ \norm{\varphi}_2\right)\left(1+\norm{\varphi}_1\right)^3\left(\norm{u}_0+\norm{\varphi}_0\right)}
$$
for some $C=C(n,\W)$.
\end{teo}
\begin{proof}
The hypothesis guaranty that \eqref{eq_est_grad_curv_bordo_Ricc_obs} holds, thus \eqref{paraObs} also holds as we explain in the remark preceding the statement of the theorem. 
The rest of the proof is the same as the proof of Theorem \ref{teo_Est_gradiente_fronteira}.
\end{proof}


In order to obtain a priori global gradient estimate the techniques introduced by Caffarelli-Nirenberg-Spruck \cite[p. 51]{CNSpruck} in the Euclidean context are used in a clever way. See other applications in the works of Nelli-Sa Earp \cite[Lemma 3.1 p. 4]{NelliRicardo} and  Barbosa-Sa Earp \cite[Lemma 5.2 p. 62]{BarbosaRicardo1998} in the hyperbolic setting.  
\begin{teo}[Global gradient estimate]\label{teo_Est_global_gradiente}
Let $\W\subset M$ be a bounded domain. If a function $u\in\cl^3(\W)\cap\cl^1(\overline{\W})$ is a solution of \eqref{operador_minimo_1} for $H\in\cl^{1}\left(\W\times\left[-\sup\limits_{\overline{\W}}\modulo{u},\sup\limits_{\overline{\W}}\modulo{u}\right]\right)$ satisfying $\Dz H\geq 0$,
then
$$
\sup_{\W}\norm{\nabla u}\leq\left({3}+\sup\limits_{\partial\W}\norm{\nabla u}\right)e^{4n\left(1+\norm{H}_1+\sup\limits_{\W}\modulo{\Ricc}\right)\sup\limits_{\W}\modulo{u}}.
$$
\end{teo}

\begin{proof}
Let $w(x)=\norm{\nabla u(x)}e^{Au(x)}$ where $A\geq 1$. Suppose $w$ attains a maximum at $x_0\in\overline{\W}$. If $x_0\in\partial\W$, then
$$w(x)\leq w(x_0) =\norm{\nabla u(x_0)}e^{Au(x_0)}.$$
Hence,
\begin{equation}\label{est_global_1}
\sup_{\W}\norm{\nabla u(x)}\leq\sup_{\partial\W}\norm{\nabla u}e^{2A\sup\limits_{\W}\modulo{u}}. 
\end{equation}

Suppose now that $x_0\in\W$ and that $\nabla u(x_0)\neq 0$. Let us define normal coordinates at $x_0$ in such a way that $\frac{\partial}{\partial x_1}\big|_{x_0}=\frac{\nabla u(x_0)}{\norm{\nabla u(x_0)}}$. Then, 
\begin{equation}\label{der_u_x0}
\Dk u(x_0)=\escalar{\Ek\big|_{x_0}}{\nabla u(x_0)}
=\norm{\nabla u(x_0)}\delta_{k1}.
\end{equation}
Besides, if $\sigma$ is the metric in this coordinate system, then
\begin{equation}\label{sigmaij}
\sigma_{ij}(x_0)=\sigma^{ij}(x_0)=\delta_{ij},
\end{equation}
\begin{equation}\label{deriv_sigmaij}
\Dk \sigma_{ij}(x_0)=\Dk \sigma^{ij}(x_0) =0,
\end{equation}
\begin{equation}\label{simb_Chris}
\Gamma_{ij}^k(x_0)=0.
\end{equation}

Observe now that the function $\tilde{w}(x)=\ln w(x)=Au(x)+\ln \norm{\nabla u(x)}$ also attains a maximum at $x_0$. 
Therefore, 
\begin{equation}\label{derk}
\Dk \tilde{w}(x_0)=A\Dk u(x_0) +\dfrac{\Dk \left(\norm{\nabla u}^2\right)(x_0)}{2\norm{\nabla u(x_0)}^2} =0,
\end{equation}
\noindent and
\begin{equation}\label{derkk}
\Dkk \tilde{w}(x_0)=A\Dkk u(x_0) - \dfrac{\left(\Dk \left(\norm{\nabla u}^2\right)(x_0)\right)^2}{2\norm{\nabla u(x_0)}^{4}}+\dfrac{\Dkk \left(\norm{\nabla u}^2\right)(x_0)}{2\norm{\nabla u(x_0)}^2}\leq 0.
\end{equation}

Let us calculate the derivatives involved in these relations. Recall first that $\nabla u(x) = \ds\sum_{i}u^i \Ei$ where
\begin{equation}\label{exp_local_grad_1}
u^i=\sum_{j=1}^n\sigma^{ij} \Dj u.
\end{equation}
Then
\begin{equation}\label{exp_local_norma_grad}
\norm{\nabla u(x)}^2=\sum_{i,j=1}^n\sigma^{ij}\Di u\Dj u
\end{equation}
and
\begin{equation}\label{Dknorma}
\Dk \left(\norm{\nabla u}^2\right)= \ds\sum_{i,j=1}^n \left( \left(\Dk\sigma^{ij}\right)\Di u \Dj u+2 \sigma^{ij}\Dki u \Dj u\right).
\end{equation}
Using \eqref{der_u_x0}, \eqref{sigmaij} and \eqref{deriv_sigmaij} one gets
\begin{equation}\label{Dknorm}
\Dk \left(\norm{\nabla u}^2\right)(x_0)=2 \norm{\nabla u(x_0)}\Dumk u(x_0).
\end{equation}

\noindent Substituting this last expression and \eqref{der_u_x0} in \eqref{derk} it follows
\begin{equation}\label{derk_2}
\Dumk u(x_0)=-A\norm{\nabla u (x_0)}^2\delta_{k1}.
\end{equation}

\noindent The combination of \eqref{Dknorm} and \eqref{derk_2} finally yields
\begin{equation}\label{Dknorm_2}
\Dk \left(\norm{\nabla u}^2\right)(x_0)=-2A\norm{\nabla u (x_0)}^3\delta_{k1}.
\end{equation}

Deriving now \eqref{Dknorma} it follows
\begin{align*}
 \Dkk \left(\norm{\nabla u}^2\right)(x)
=&\ds\sum_{i,j=1}^n \left(\left(\Dkk\sigma^{ij}\right)\Di u \Dj u + \left(\Dk\sigma^{ij}\right)\Dk\left(\Di u \Dj u\right)\right.\\
 &\left.+2\left( \left(\Dk\sigma^{ij}\right) \Dki u\Dj u + \sigma^{ij} \Dkki u \Dj u +\sigma^{ij}\Dki u\Dkj u\right)\right).
\end{align*}
Relations \eqref{der_u_x0}, \eqref{sigmaij} and \eqref{deriv_sigmaij} are used again to obtain
\begin{equation}\label{Dkk_2}
\Dkk \left(\norm{\nabla u}^2\right)(x_0)=\norm{\nabla u(x_0)}^2 \left(\Dkk\sigma^{11}\right)(x_0) +2\norm{\nabla u(x_0)}\Dkkum u(x_0) +2 \ds\sum_{i=1}^n (\Dki u(x_0))^2.
\end{equation} 
Also, from \eqref{sigmaij}, \eqref{deriv_sigmaij} and \eqref{simb_Chris} it can be seen that
$$
\Dkk\sigma^{11}(x_0)=-\Dkk\sigma_{11}(x_0)=-2\escalar{\conex_{\Ek}\conex_{\Ek}\Eum}{\Eum}.
$$
Therefore, 
{\small
\begin{equation}\label{dkknorma_2}
\begin{split}
\Dkk \left(\norm{\nabla u}^2\right)(x_0)=2\left(-\norm{\nabla u (x_0)}^2\escalar{\conex_{\Ek}\conex_{\Ek} \Eum}{\Eum} +\norm{\nabla u(x_0)}\ds \Dkkum u(x_0) +\ds\sum_{i=1}^n\left(\Dki u(x_0)\right)^2\right).
\end{split}
\end{equation}}%

Using \eqref{Dknorm_2} and \eqref{dkknorma_2} inequality \eqref{derkk} becomes
{
\[
\begin{split}
A\Dkk u(x_0)-2A^2\norm{\nabla u (x_0)}^2\delta_{k1}+\ds \dfrac{\Dkkum u(x_0)}{\norm{\nabla u(x_0)}}-\escalar{\conex_{\Ek}\conex_{\Ek} \Eum}{\Eum}
+\dfrac{\ds\sum_{i=1}^n\left(\Dki u(x_0)\right)^2}{\norm{\nabla u (x_0)}^2} \leq 0.
\end{split}
\]}%
Since \eqref{derk_2} holds it can be inferred that
\begin{equation}\label{est_Dumumum}
\Dumumum u(x_0)\leq 2A^2\norm{\nabla u (x_0)}^3+\norm{\nabla u(x_0)}\escalar{\conex_{\Eum}\conex_{\Eum} \Eum}{\Eum}.
\end{equation}
\noindent and 
\begin{equation}\label{est_Dkkum}
\Dkkum u(x_0)\leq -A\Dkk u(x_0)\norm{\nabla u (x_0)}+\norm{\nabla u(x_0)}\escalar{\conex_{\Ek}\conex_{\Ek} \Eum}{\Eum} \ \mbox{ if }\  k>1.
\end{equation}

In the sequel we evaluate at $x_0$ the mean curvature equation \eqref{operador_minimo_1_coord}.
First, recall that 
\begin{equation}\label{hessianohij_coord}
\Hessij u(x)=\Hess u(x){\left(\Ei,\Ej\right)}=\Dij u-\ds\sum_{k=1}^n\Gamma_{ij}^k \Dk u,
\end{equation}
\begin{equation}\label{laplacianof}
\Delta u(x)=
\ds\sum_{ij}\sigma^{ij}\Hessij  u (x).
\end{equation}

\medskip

\noindent Using \eqref{der_u_x0}, \eqref{sigmaij} and \eqref{simb_Chris} it can easily be seen that the quantities \eqref{exp_local_grad_1}, \eqref{hessianohij_coord} and \eqref{laplacianof} at $x_0$ have the values
\begin{equation}\label{ui}
u^i(x_0)=\Di u(x_0)=\norm{\nabla u(x_0)}\delta_{i1},
\end{equation}
\begin{equation}\label{hess_x0}
\Hessij u(x_0)=\Dij u(x_0),
\end{equation}
\begin{equation}\label{laplac_x0}
\Delta u(x_0)=\ds\sum_{i=1}^n \Dii u(x_0).
\end{equation}

\noindent Using these expressions the mean curvature equation \eqref{operador_minimo_1_coord} at $x_0$ takes the form
\begin{align*}
 nH_0 W_0^3=W_0^2 \Delta u(x_0)-\norm{\nabla u(x_0)}^2\Dumum u(x_0)=W_0^2\ds\sum_{i>1}\Dii u(x_0) + \Dumum u(x_0),
\end{align*}
where $H_0=H(x_0,u(x_0))$ and $W_0=\sqrt{1+\norm{\nabla u(x_0)}^2}$. Using \eqref{derk_2} again it follows
\begin{equation}\label{equ_curv_media_MxR_Delta_x0}
\ds\sum_{i>1}\Dii u(x_0) = nH_0W_0+\dfrac{A\norm{\nabla u(x_0)}^2}{W_0^2}.
\end{equation}

Finally let us differentiate \eqref{operador_minimo_1_coord} with respect to $x_1$ and evaluate at $x_0$. We have 
\begin{equation}\label{der_MCE}
\begin{split}
\left(\Dum \left(W^2\right)\right)\Delta u+W^2\left(\Dum\Delta u\right)-2\ds\sum_{i,j=1}^nu^i\left(\Dum u^j\right) \Hessij u -\sum_{i,j=1}^n u^i u^j \left(\Dum\Hessij u\right)
\\
=n(\Dum H+\Dz H\Dum u)W^3 + nH\left(\Dum \left(W^3\right)\right).
\end{split}
\end{equation}
Expression \eqref{Dknorm_2} immediately gives 
\begin{equation}\label{dW2}
\Dum \left(W^2 \right)(x_0)=\Dum\left(\norm{\nabla u}^2\right)(x_0)=-2A\norm{\nabla u (x_0)}^3,
\end{equation}
\begin{equation}\label{dW3}
\Dum \left(W^3\right) (x_0)=\frac{3}{2}W_0\Dum \left(W^2\right)(x_0)=-3AW_0\norm{\nabla u (x_0)}^3.
\end{equation}

Deriving \eqref{exp_local_grad_1} 
and using \eqref{sigmaij}, \eqref{deriv_sigmaij} and \eqref{derk_2} one gets
\begin{equation}\label{dui}
\Dum u^i(x_0)=\Dumi u(x_0) =-A\norm{\nabla u(x_0)}^2\delta_{i1}.
\end{equation}

On the other hand, from \eqref{hessianohij_coord} we deduce
{
\begin{align*}
\Dum\Hessij u(x)
=&\Dumij u(x) - \escalar{\conex_{\Eum} \nabla u}{\conex_{\Ei}\Ej}-\escalar{\nabla u(x)}{\conex_{\Eum}\conex_{\Ei}\Ej}.
\end{align*}}%
Consequently, since \eqref{simb_Chris} holds,
\begin{equation}\label{DumHessiano}
\Dum\Hessij u(x_0)=\Dumij u(x_0) - \escalar{\conex_{\Eum}\conex_{\Ei}\Ej}{\nabla u(x_0)}.
\end{equation}
Finally, deriving \eqref{laplacianof}
and using \eqref{sigmaij}, \eqref{deriv_sigmaij} and \eqref{DumHessiano} one can infer
\begin{equation}\label{DumLaplaciano}
\Dum\Delta u(x_0)=\ds\sum_{i=1}^n\left(\Dumii u(x_0) - \escalar{\conex_{\Eum}\conex_{\Ei}\Ei}{\nabla u(x_0)}\right).
\end{equation}

\bigskip

Substituting \eqref{der_u_x0}, \eqref{ui}, \eqref{hess_x0}, \eqref{dW2}, \eqref{dW3}, \eqref{dui}, \eqref{DumHessiano} and \eqref{DumLaplaciano} in \eqref{der_MCE} we obtain 
{
\begin{align*}
&n\Dum H(x_0) W_0^3 + n \Dz H(x_0)\norm{\nabla u(x_0)} W_0^3-3nA H_0 W_0\norm{\nabla u(x_0)}^3\\[1em]
=&-2A\norm{\nabla u(x_0)}^3\Delta u(x_0)+W_0^2\ds\sum_{i=1}^n\left(\Dumii u(x_0) - \escalar{\conex_{\Eum}\conex_{\Ei}\Ei}{\nabla u(x_0)}\right)\\
&+2A\norm{\nabla u(x_0)}^3\Dumum u(x_0)-\norm{\nabla u(x_0)}^2\left(\Dumumum u(x_0) - \escalar{\conex_{\Eum}\conex_{\Eum}\Eum}{\nabla u(x_0)}\right)\\[1em]
=&-2A\norm{\nabla u(x_0)}^3\ds\sum_{i>1}\Dii u(x_0)+{\Dumumum u(x_0) - \escalar{\conex_{\Eum}\conex_{\Eum}\Eum}{\nabla u(x_0)}}\\
&+W_0^2\ds\sum_{i>1}\left(\Dumii u(x_0) - \escalar{\conex_{\Eum}\conex_{\Ei}\Ei}{\nabla u(x_0)}\right).
\end{align*}
}
Using \eqref{est_Dumumum}, \eqref{est_Dkkum}, \eqref{equ_curv_media_MxR_Delta_x0} and recalling that $\Dz H\geq 0$ we derive
\begin{align*}
&n\Dum H(x_0) W_0^3 -3nA H_0 W_0\norm{\nabla u(x_0)}^3\\[1em]
\leq&-2A\norm{\nabla u(x_0)}^3\ds\sum_{i>1}\Dii u(x_0)+2A^2\norm{\nabla u(x_0)}^3\\
&+W_0^2\norm{\nabla u(x_0)}\ds\sum_{i>1}\left(-A\Dii u(x_0) +\escalar{\conex_{\Ei}\conex_{\Ei}\Eum}{\Eum} - \escalar{\conex_{\Eum}\conex_{\Ei}\Ei}{\Eum}\right) \\[1em]
=&-A\norm{\nabla u(x_0)}\left(W_0^2+2\norm{\nabla u(x_0)}^2\right)\ds\sum_{i>1}\Dii u(x_0)+2A^2\norm{\nabla u(x_0)}^3 \\&+\norm{\nabla u(x_0)}W_0^2\ds\sum_{i>1}\escalar{R\left(\Ei,\Eum\right)\Ei}{\Eum}\\[1em]
=&-A\norm{\nabla u(x_0)}\left(1+3\norm{\nabla u(x_0)}^2\right)\left(nH_0W_0+\dfrac{A\norm{\nabla u(x_0)}^2}{W_0^2} \right)\\
& +2A^2\norm{\nabla u(x_0)}^3 -\norm{\nabla u(x_0)}W_0^2\Ricc_{x_0}\left(\Eum\right)\\[1em]
=&-A\norm{\nabla u(x_0)}nH_0W_0\left(1+3\norm{\nabla u(x_0)}^2\right)
-\dfrac{A^2\norm{\nabla u(x_0)}^3}{W_0^2}\left(1+3\norm{\nabla u(x_0)}^2\right) \\
& +2A^2\norm{\nabla u(x_0)}^3 -\norm{\nabla u(x_0)}W_0^2\Ricc_{x_0}\left(\Eum\right).
\end{align*}
Hence,
\begin{align*}
\dfrac{A^2\norm{\nabla u(x_0)}^3\left(\norm{\nabla u(x_0)}^2-1\right)}{W_0^2} \leq & - A n H_0 W_0\norm{\nabla u(x_0)}- n\Dum H(x_0) W_0^3 -\norm{\nabla u(x_0)}W_0^2\Ricc_{x_0}\left(\Eum\right)\\
\leq & A n h_0 W_0\norm{\nabla u(x_0)} + n h_1 W_0^3 +\norm{\nabla u(x_0)}W_0^2 R,
\end{align*}
where $h_0=\sup\limits_{\W\times\left[-\norm{u}_0,\norm{u}_0\right]}\modulo{H}$, $h_1=\sup\limits_{\W\times\left[-\norm{u}_0,\norm{u}_0\right]}\left(\norm{\nabla_x H}+\Dz H\right)$
and $R=\sup\limits_{\W}\modulo{\Ricc}$. 
Therefore,
\begin{align*}
\dfrac{A^2\norm{\nabla u(x_0)}^3\left(\norm{\nabla u(x_0)}^2-1\right)}{W_0^5}\leq A n h_0  + n h_1 + R \leq A n \left( h_0  + h_1 + R \right) 
\end{align*}
since $W_0^2> W_0 >\norm{\nabla u(x_0)}$ and $A,n>1$. Choosing $A=2n\left(1+\norm{H}_1+R\right)$ it follows
$$ \dfrac{\norm{\nabla u(x_0)}^3\left(\norm{\nabla u(x_0)}^2-1\right)}{W_0^5}\leq \dfrac{n}{A} \left(\norm{H}_1+R\right)<\frac{1}{2}  ,$$
which implies
$$ \norm{\nabla u(x_0)}<{3}.$$
As a consequence,
\begin{equation}\label{est_global_grad_interior}
\sup_{\W}\norm{\nabla u}\leq{3}e^{2A\sup\limits_{\W}\modulo{u}}. 
\end{equation}

The combination of \eqref{est_global_1} with \eqref{est_global_grad_interior} yields the desired estimate.
\end{proof}

\begin{obs}
A related global gradient estimate was obtained independently in \cite[Prop. 2.2 p. 5]{2018miriam}. 
\end{obs}
\section{Proof of the existence theorems} 

\bigskip

\noindent\textit{Proof of the main theorem (Theorem \ref{T_exist_Ricci}).}
Let $\W\subset M$ with $\partial\W$ of class $\cl^{2,\alpha}$ for some $\alpha\in(0,1)$ and $\varphi\in\cl^{2,\alpha}(\overline{\W})$. 
Elliptic theory assures that the solvability of the Dirichlet problem \eqref{ProblemaP}  strongly depends on $\cl^{1}$ a priori estimates for the family of related problems 
\begin{equation}\tag{$P_{\tau}$}\label{ProblemaPsigma}
\left\{\begin{array}{l}
 \diver\left(\dfrac{\nabla u}{W}\right)  = \tau nH(x,u)\ \mbox{ in }\ \W, \\
 \phantom{\diver\left(\dfrac{\nabla u}{W}\right)}      u  = \tau \varphi \ \mbox{ in }\  \partial\Omega,
\end{array}\right.
\end{equation}
not depending on $\tau$ or $u$. 

Let $u$ be a solution of problem \eqref{ProblemaPsigma} for arbitrary $\tau\in[0,1]$. 
Let $w=\phi\circ d + \sup\limits_{\partial\W}\modulo{\varphi}$ as in the proof of Theorem \ref{teo_Est_altura}. Then
$$ u\leq \sup\limits_{\partial\W} \modulo{\tau\varphi}\leq \sup\limits_{\partial\W} \modulo{\varphi}=w\ \mbox{ on } \ \partial\W.$$
As before, let $\W_0$ be the biggest open subset of $\W$ having the unique nearest point property. Let $x\in\W_0$ and $y=y(x)\in\partial\W$ the nearest point to $x$. Once \eqref{eq_Hxw} holds and $\tau\in[0,1]$ we have that 
$$ \mp n \tau {H(x,\pm w)} \leq n \tau\modulo{H(x,\varphi(y))} \leq n \modulo{H(x,\varphi(y))}. $$
From \eqref{est_Mw_est_alt} we have
\begin{align*}
\pm\Q_{\tau}(\pm w)=\M w\mp n\tau H(x,\pm w)(1+\phi'^2)^{3/2}
\leq 0.
\end{align*}
Proceeding as in the proof of Theorem \ref{teo_Est_altura}, we get that $w$ and $-w$ are supersolution and subsolution in $\W_0$, respectively, for the problem \eqref{ProblemaPsigma}. This provides a priori height estimate for any solution of the problems \eqref{ProblemaPsigma} independently of $\tau$.

{%
In order to prove that Theorem \ref{teo_Est_gradiente_fronteira} provides a priori gradient estimate for the solutions of the related problems \eqref{ProblemaPsigma} let us define $w^{\pm}_\tau=\pm \phi\circ d + \tau\varphi$. Making an examination of the proof of this theorem it can easily be seen that, on account of assumptions \eqref{cond_H_Ricci_exist} and \eqref{StrongSerrinCondition_mainThm}, the geometric condition \eqref{est_usar_hiperb_0} holds for $\W$. Using again the strong Serrin condition \eqref{StrongSerrinCondition_mainThm} one has
\begin{equation}\label{est_usar_hiperb_3}
\Delta d(x) \leq -n \tau \modulo{H(y,\tau\varphi(y))} \ \forall \ x\in\W_a. 
\end{equation}
Replacing \eqref{est_usar_hiperb} by \eqref{est_usar_hiperb_3} we can derive in analogous way
$$\pm \Q_\tau \left(w^{\pm}_\tau\right)<\nu\psi'+\psi''=0,$$ 
where $\nu$ is the same defined in \eqref{nu}. Proceeding exactly as before we obtain 
$$\norm{\nabla u(y)}\leq\tau\norm{\nabla \varphi(y)} + \psi'(0)\leq\norm{\nabla \varphi(y)} + \psi'(0),$$
which yields the same estimate \eqref{EstGradFront}, which is independent of $\tau$, for all solutions of the related problems \eqref{ProblemaPsigma}. 
}

On the other hand, elliptic theory guarantees that any solution $u$ of the related problems \eqref{ProblemaPsigma} belongs to $\cl^3(\W)$. Hence, Theorem \ref{teo_Est_global_gradiente} can be applied to obtained the desired a priori global gradient estimate independently of $\tau$ and $u$. 

The existence of a solution $u\in\cl^{2,\alpha}(\overline{\W})$ for the Dirichlet problem \eqref{ProblemaP} is obtained applying the Leray-Schauder fixed point theorem in an usual way (see {\cite[Th. 11.4 p. 281]{GT}}). Uniqueness follows from the maximum principle, in view of the assumption $\Dz H\geq 0$.
\hfill \qedsymbol

\bigskip

\noindent{\it Proof of Theorem \ref{exist_hiperbolicoHmenor1}.}
We first recall that in $\HH^n\times\R$ there exists an entire vertical graph of constant mean curvature $\frac{n-1}{n}$. Explicit formulas were given by B\'erard-Sa Earp {\cite[Th. 2.1 p. 22]{BerardRicardo}}. The a priori height estimate for the solutions of the related problems \eqref{ProblemaPsigma} follows directly from the convex hull lemma \cite[Prop. 3.1 p. 41]{BerardRicardo}. 

The rest of the proof is the same as before, being that the a priori boundary gradient estimate follows from Theorem \ref{teo_Est_gradiente_fronteira_hiperbólico}. 
\hfill\qedsymbol

\bigskip

\noindent{\it Proof of Theorem \ref{teo_GalvezLozano}.}
Under the hypothesis on $M$ and $\W$, Galvez-Lozano \cite[Th. 6 p. 12]{Galvez2015} proved the existence of a vertical graph over $\W$ with constant mean curvature $\frac{n-1}{n}$ and zero boundary data. As a matter of fact, such a graph constitutes a barrier for the solutions of the related problems \eqref{ProblemaPsigma}. 

On the other hand, the strong Serrin condition trivially holds since, for $y\in\partial\W$, 
$$(n-1)\Hc_{\partial\W}(y)> (n-1)c > n-1 \geq n \sup\limits_{\W\times\R}\modulo{H(x,z)}.$$
Besides, 
$$ \Ricc_x\geq -(n-1)c^2 .$$
Accordingly, the boundary gradient estimate follows from our Theorem \ref{teo_Est_gradiente_fronteira_paraGalvez}. 
\hfill \qedsymbol

\bibliographystyle{amsplainnovo}
\bibliography{bibliografia}

\newpage \ 
\vfill

\noindent Yunelsy N. Alvarez\\
Departamento de Matemática\\
Pontifícia Universidade Católica do Rio de Janeiro\\
Rio de Janeiro, Brazil, CEP 22451-900 \smallskip\\
Present Address: \\
Departamento de Matemática \\
Instituto de Matemática e Estadística\\
Universidade de São Paulo\\
São Paulo, Brazil,  CEP 05508-090 \\
Email addresses: ynapolez@gmail.com; ynalvarez@usp.br

\bigskip

\noindent Ricardo Sa Earp\\
Departamento de Matemática\\
Pontifícia Universidade Católica do Rio de Janeiro\\
Rio de Janeiro, CEP 22451-900, Brazil\\
Email address: rsaearp@gmail.com\\
\end{document}